\documentclass{amsart}

\numberwithin{equation}{section}

\usepackage[centertags]{amsmath}
\usepackage{amsfonts}
\usepackage{amssymb}
\usepackage{amsthm}
\usepackage{amscd}
\usepackage{hyperref}
\usepackage{algorithmic, algorithm}
\usepackage[all]{xy}

\usepackage{color}
\usepackage[dvipsnames]{xcolor}

\usepackage{mathrsfs} 

\DeclareFontFamily{OT1}{rsfs}{}
\DeclareFontShape{OT1}{rsfs}{n}{it}{<-> rsfs10}{}
\DeclareMathAlphabet{\mathscr}{OT1}{rsfs}{n}{it}

\newcommand{\comment}[1]{}

\renewcommand{\P}{\mathbf{P}}

\newcommand{\F}{\mathbf{F}}

\DeclareFontEncoding{OT2}{}{} 

\DeclareMathOperator{\Gal}{Gal}

\theoremstyle{plain} 
\newtheorem{thm}{Theorem}[section] 
\newtheorem{prop}[thm]{Proposition}
\newtheorem{cor}[thm]{Corollary}
\newtheorem{lem}[thm]{Lemma}

\theoremstyle{definition} 
\newtheorem{defn}[thm]{Definition} 
\theoremstyle{remark} 

\newcounter{tasknumber}
\newcommand{\task}[2][]{%
  \addtocounter{tasknumber}{1}%
  \begin{center}%
  \framebox[1.1\width]{\begin{minipage}{0.9\textwidth}%
  \textbf{Task \arabic{tasknumber}} \textit{\if!#1(unassigned)!\else (#1)\fi}: {#2}%
  \end{minipage}}%
  \end{center}%
}

\newcounter{assumptionnumber}
\newcommand{\assumption}[2][]{%
  \addtocounter{assumptionnumber}{1}%
  \begin{center}%
  \framebox[1.1\width]{\begin{minipage}{0.9\textwidth}%
  \textbf{Assumption \arabic{assumptionnumber}} \textit{\if!#1!\else (#1)\fi}: {#2}%
  \end{minipage}}%
  \end{center}%
}

\newcommand{\authnote}[2][]{\noindent {\if!#1!  {\bf TODO} \else {\small \bf #1} \fi: #2}}

\newcommand{\algclosure}{{\overline{\F}_q}}

\begin{document}

\title[A new perspective on the powers of two descent]{A new perspective on the powers of two descent\\ for discrete logarithms in finite fields}


\author{Thorsten Kleinjung}
\author{Benjamin Wesolowski}
\address{\'Ecole Polytechnique F\'ed\'erale de Lausanne, EPFL IC LACAL, Switzerland}

\begin{abstract}
A new proof is given for the correctness of the powers of two descent method for computing discrete logarithms.
The result is slightly stronger than the original work, but more importantly we provide a unified geometric argument, eliminating the need to analyse all possible subgroups of $\mathrm{PGL}_2(\F_q)$.
Our approach sheds new light on the role of $\mathrm{PGL}_2$, in the hope to eventually lead to a complete proof that discrete logarithms can be computed in quasi-polynomial time in finite fields of fixed characteristic.
\end{abstract}

\maketitle

\section{Introduction} 

In this paper we prove the following result.
\begin{thm}\label{thm:maintheorem}
Given a prime power $q$, a positive integer $d$, 
coprime polynomials $h_0$ and $h_1$ in $\F_{q^d}[x]$ of degree at most two, and an irreducible degree $\ell$ factor~$I$ of ${h_1x^q - h_0}$, the discrete logarithm problem in $\F_{q^{d\ell}} \cong \F_{q^d}[x]/(I)$ can be solved in expected time $q^{\log_2\ell + O(d)}$.
\end{thm}

It was originally proven in~\cite{GKZ} when $q > 61$, $q$ is not a power of $4$, and ${d \geq 18}$.
Even though we eliminate these technical conditions, the main contribution is the new approach to the proof.
The theorem represents the state of the art of provable quasi-polynomial time algorithms for the discrete logarithm problem (or DLP) in finite fields of fixed characteristic. The obstacle separating Theorem~\ref{thm:maintheorem} from a full provable algorithm for DLP is the question of the existence of a good field representation: polynomials $h_0$, $h_1$ and $I$ for a small $d$. A direction towards a full provable algorithm would be to find analogues of this theorem for other field representations, but this may require in the first place a good understanding of why Theorem~\ref{thm:maintheorem} is true.

The integers $q$, $d$ and $\ell$, and the polynomials $h_0, h_1$ and $I$ are defined as in the above theorem for the rest of the paper.
The core of that result is Proposition~\ref{prop:onestepdescent} below, which essentially states that elements of $\F_{q^{d\ell}}$ represented by a \emph{good} irreducible polynomial in $\F_{q^d}[x]$ of degree $2m$ can be rewritten as a product of \emph{good} irreducible polynomials of degrees dividing $m$ --- a process called \emph{degree two elimination}, first introduced for $m=1$ in~\cite{GGMZ}.

\begin{defn}[Traps and good polynomials]
An element $\tau \in \algclosure$ for which $[\F_{q^d}(\tau) : \F_{q^d}]$ is an even number $2m$ and $h_1(\tau) \neq 0$ is called
\begin{enumerate}
\item a \emph{degenerate trap root} if $\frac{h_0}{h_1}(\tau) \in \F_{q^{dm}}$,
\item a \emph{trap root of level $0$} if it is a root of $h_1x^q-h_0$, or
\item a \emph{trap root of level $dm$} if it is a root of $h_1x^{q^{dm+1}} - h_0$.
\end{enumerate}
Analogously, a polynomial in $\algclosure[x]$ that has a trap root is called a \emph{trap}.
A polynomial is \emph{good} if it is not a trap.
\end{defn}

\begin{prop}[Degree two elimination]\label{prop:onestepdescent}
Given an extension $k/\F_{q^d}$ of degree $m$ such that $dm \geq 23$, and a good irreducible quadratic polynomial $Q \in k[x]$, there is an algorithm which finds a list of good linear polynomials $(L_0,\dots,L_n)$ in $k[x]$ such that $n \leq q+1$ and
$$Q \equiv h_1L_0^{-1}\cdot\prod_{i = 1}^nL_i \mod I,$$
and runs in expected polynomial time in $q$, $d$ and $m$.
\end{prop}

The difficulty of proving Theorem~\ref{thm:maintheorem} lies mostly in Proposition~\ref{prop:onestepdescent}. We recall briefly in Section~\ref{sec:propimpliesthm} how the proposition implies the theorem. The main contribution of the present paper is a new proof of Proposition~\ref{prop:onestepdescent}, which hopefully provides a better understanding of the degree two elimination method, the underlying geometry, and the role of traps. The action of $\mathrm{PGL}_2$ on the polynomial $x^q - x$ became a crucial ingredient in the recent progress on the discrete logarithm problem for fields of small characteristic, since~\cite{Joux13} (and implicitly in~\cite{GGMZ}). While the proof in~\cite{GKZ} resorted to an intricate case by case analysis enumerating through all possible subgroups of $\mathrm{PGL}_2(\F_q)$, we provide a unified geometric argument, shedding new light on the role of $\mathrm{PGL}_2$.

\subsection{Degree two elimination algorithm} \label{sec:overview}
The key observation allowing degree two elimination is that a polynomial of the form $\alpha x^{q+1} + \beta x^q + \gamma x + \delta$ has a high chance to split completely over its field of definition. Furthermore, we have the congruence
\begin{equation}\label{eq:maincong}
\alpha x^{q+1} + \beta x^q + \gamma x + \delta \equiv h_1^{-1}(\alpha xh_0 + \beta h_0 + \gamma xh_1 + \delta h_1) \mod I,
\end{equation}
and the numerator of the right-hand side has degree at most $3$. Consider the $\algclosure$-vector space $V$ spanned by $x^{q+1},x^q,x$ and $1$ in  $\algclosure[x]$, and the linear subspace
$$V_Q = \{\alpha x^{q+1} + \beta x^q + \gamma x + \delta \in V \mid \alpha xh_0 + \beta h_0 + \gamma x h_1 + \delta h_1 \equiv 0 \mod Q\}.$$
As long as $Q$ is a good irreducible polynomial, $V_Q$ is of dimension two. The algorithm simply consists in sampling uniformly at random elements $f \in V_Q(k)$ (or equivalently in its projectivisation $\P^1_Q(k)$) until $f$ splits completely over $k$ into good linear polynomials $(L_1,\dots,L_{\deg f})$. Since $f \in V_Q$, the polynomial $Q$ divides the numerator of the right-hand side of \eqref{sec:overview}, and the quotient is a polynomial $L_0$ of degree at most $1$. The algorithm returns $(L_0,\dots,L_{\deg f})$.

To prove that the algorithm terminates in expected polynomial time, we need to show that a random polynomial in $V_Q(k)$ has good chances to split into good linear polynomials over $k$. In this paper, we prove this by constructing a morphism $C \rightarrow \P^1_Q$ where $C$ is an absolutely irreducible curve defined over $k$, such that the image of any $k$-rational point of $C$ is a polynomial that splits completely over $k$. This construction is the object of Section~\ref{sec:ramifications}. The absolute irreducibility implies that $C$ has a lot of $k$-rational points, allowing us to deduce that a lot of polynomials in $\P^1_Q(k)$ split over $k$. This is done in Section~\ref{sec:proofmainprop}.

\subsection{Proof of Theorem~\ref{thm:maintheorem}} \label{sec:propimpliesthm}
We briefly explain in this section how Proposition~\ref{prop:onestepdescent} implies Theorem~\ref{thm:maintheorem}.
Consider the factor base $$\mathfrak F = \{f \in \F_{q^d}[x] \mid \deg f \leq 1, f \neq 0\} \cup \{h_1\}.$$
First, the following proposition extends the degree two elimination to a full descent algorithm from any polynomial down to the factor base.
\begin{prop}\label{prop:descent}
Suppose $d \geq 23$.
Given a polynomial $F \in \F_{q^d}[x]$, there is an algorithm that finds integers $(\alpha_f)_{f \in \mathfrak F}$ such that
$$F \equiv \prod_{f \in \mathfrak F}f^{\alpha_f} \mod I,$$
and runs in expected time $q^{\log_2\ell + O(d)}$.
\end{prop}

\begin{proof}
This is essentially the \emph{zigzag} descent presented in~\cite{GKZ}. We recall the main idea for the convenience of the reader. First, one finds a good irreducible polynomial $G \in \F_{q^d}[x]$ of degree $2^e$ such that $F \equiv G \mod I$ (this can be done for $e = \lceil\log_2(4\ell+1)\rceil$, see~\cite[Th. 5.1]{Daqing} and~\cite[Lem. 2]{GKZ}).
Over the extension $\F_{q^{d2^{e-1}}}$, the polynomial $G$ splits into $2^{e-1}$ good irreducible quadratic polynomials, all conjugate under $\Gal(\F_{q^{d2^{e-1}}}/\F_{q^{d}})$. Let $Q$ be one of them, and apply the algorithm of Proposition~\ref{prop:onestepdescent} to rewrite $Q$ in terms of linear polynomials $(L_0,\dots,L_n)$ in $\F_{q^{d2^{e-1}}}[x]$ and $h_1$.
For any index $i$, let $L_i'$ be the product of all the conjugates of $L_i$ in the extension $\F_{q^{d2^{e-1}}}/\F_{q^{d}}$. Then,
$$F \equiv h_1^{2^{e-1}}L_0'^{-1}\cdot\prod_{i = 1}^nL_i' \mod I,$$
and each $L_i'$ factors into good irreducible polynomials of degree a power of 2 at most $2^{e-1}$. The descent proceeds by iteratively applying this method to each $L_i'$ until all the factors are in the factor base $\mathfrak F$.
\end{proof}

Then, as in~\cite[Sec. 2]{GKZ}, the descent algorithm of Proposition~\ref{prop:descent} can be used to compute discrete logarithms, following ideas from~\cite{EG02} and~\cite{diem_2011}. To compute the discrete logarithm of an element $h$ in base $g$, the idea is to collect relations between $g$, $h$, and elements of the factor base by applying the descent algorithm on $g^\alpha h^\beta$ for a few uniformly random exponents $\alpha$ and $\beta$.
That proves Theorem~\ref{thm:maintheorem} for $d \geq 23$. To remove the condition on $d$, suppose that $d \leq 22$, and let $d' \leq 44$ be the smallest multiple of $d$ larger than $22$. Let $I'$ be an irreducible factor of $I$ in $\F_{q^{d'}}[x]$.
The DLP can be solved in expected time $q^{\log_2(\deg I') + O(d')} = q^{\log_2 \ell + O(1)}$ in $\F_{q^{d'}}[x]/(I')$, and therefore also in the subfield $\F_{q^d}[x]/(I)$.

\section{The action of $\mathrm{PGL}_2$ on $x^q-x$} 
As already mentioned, a crucial fact behind degree two elimination is that a polynomial of the form $\alpha x^{q+1} + \beta x^q + \gamma x + \delta$ has a high chance to split completely over its field of definition. This fact is closely related to the action of $2\times 2$ matrices on such polynomials.
\begin{defn} We denote by $\star$ the action of invertible $2 \times 2$ matrices on univariate polynomials defined as follows:
$$\left(\begin{matrix}a&b \\ c & d\end{matrix}\right)\star f(x) = (cx+d)^{\deg f}f\left(\frac{ax+b}{cx+d}\right).$$
\end{defn}

Consider the $\algclosure$-vector subspace $V$ spanned by $x^{q+1},x^q,x,$ and $1$ in $\algclosure[x]$.
The above action induces an action of the group $\mathrm{PGL}_2$ on the projective space $\P(V)$, which we also write $\star$.
Parameterizing the polynomials in $\P(V)$ as $\alpha x^{q+1} + \beta x^q + \gamma x + \delta$, let $S$ be the quadratic surface in $\P(V)$ defined by the equation $\alpha\delta = \beta\gamma$. This surface is the image of the morphism
\begin{alignat*}{1}\label{eq:definitionS}
\psi: \P^1 \times \P^1 &\longrightarrow \P(V) : 
(a,b) \longmapsto (x-a)(x-b)^q.
\end{alignat*}
Note that to avoid heavy notation, everything is written affinely, but we naturally have $\psi(\infty,b) = (x-b)^q$, $\psi(a,\infty) = x-a$ and $\psi(\infty,\infty) = 1$. More generally, we say that $f(x) \in V$ has a root of degree $n$ at infinity if $f$ is of degree $q+1-n$.
Now, the following lemma shows that apart from the surface $S$, the polynomials of $\P(V)$ form exactly one orbit for $\mathrm{PGL}_2$. 

\begin{lem}\label{lem:pgl2orbit}
We have $\P(V) \setminus S = \mathrm{PGL}_2 \star (x^q-x)$.
\end{lem}
\begin{proof}
First notice that both $S$ and $\P(V) \setminus S$ are closed under the action of $\mathrm{PGL}_2$. In particular, $\mathrm{PGL}_2 \star (x^q-x) \subseteq \P(V) \setminus S$.
Let $f(x) \in \P(V) \setminus S$. Suppose by contradiction that $f(x)$ has a double root $r \in \P^1$, and let $g \in \mathrm{PGL}_2$ be a linear transformation sending $0$ to $r$. The polynomial $g \star f(x)$ has a double root at $0$, so has no constant or linear term, and must be of the form $\alpha x^{q+1} + \beta x^q$, so it is in $S$, a contradiction. Therefore $f(x)$ has $q+1$ distinct roots. Let $g \in \mathrm{PGL}_2$ send $0$, $1$ and $\infty$ to three of these roots. Then, $g \star f(x)$ has a root at $0$ and at $\infty$ so is of the form $\beta x^q + \gamma x$, and since it also has a root at $1$, it can only be $x^q - x$.
\end{proof}

This result implies that most polynomials of $\P(V)$ are of the form $g \star (x^q-x)$, which splits completely over the field of definition of the matrix $g$.

\section{The role of traps} 
Consider a finite field extension $k/\F_{q^d}$ of degree $m$.
Let $Q$ be an irreducible quadratic polynomial in $k[x]$ coprime to $h_1$. Let ${a_1}$ and ${a_2}$ be the roots of $Q$ in~$\algclosure$.
The degree two elimination aims at expressing $Q$ modulo $h_1x^q - h_0$ as a product of linear polynomials.
To do so, we study a variety $\P^1_Q \subset \P(V)$ parameterizing polynomials that can possibly lead to an elimination of $Q$ (i.e., such that $Q$ divides the right hand side of~\eqref{eq:maincong}).
In this section, we define $\P^1_Q$ and show how the notion of traps and good polynomials determine how it intersects the surface $S$ from Lemma~\ref{lem:pgl2orbit}.

Recall that $V$ is the $\algclosure$-vector subspace $V$ spanned by $x^{q+1},x^q,x,$ and $1$ in $\algclosure[x]$. Consider the linear map 
\begin{equation}\label{eq:morphxqtoh}
\varphi : V \longrightarrow \algclosure[x][h_1^{-1}] :
\begin{cases}
1 &\longmapsto 1,\\
x &\longmapsto x,\\
x^q &\longmapsto {h_0}/{h_1},\\
x^{q+1} &\longmapsto x{h_0}/{h_1}.\\
\end{cases}
\end{equation}
We want $\P^1_Q$ to parameterise the polynomials $f \in V$ such that $\varphi(f)$ is divisible by~$Q$.
For any $P \in \algclosure[x]$ coprime with $h_1$, write $\varphi_P = \pi_P \circ \varphi$ where $\pi_P : \algclosure[x][h_1^{-1}] \rightarrow \algclosure[x]/P$ is the canonical projection. 
We can now define $\P^1_Q$ as
\begin{equation}\label{eq:defnPQ}
\P^1_Q = \P(\ker \varphi_Q).
\end{equation}
The variety $\P^1_Q$ is the intersection of the two planes $\P(\ker \varphi_{x-{a_1}})$ and $\P(\ker \varphi_{x-{a_2}})$.

\begin{lem}\label{lem:nondegenerate}
If $Q$ is not a degenerate trap, then ${|(\P^1_Q \cap S)(\algclosure)| = 2}$, and these two points are of the form $\psi(a_1,b_1)$ and $\psi(a_2,b_2)$, with $a_1 \neq a_2$ and $b_1 \neq b_2$.
\end{lem}
\begin{proof}
For $a \in \{{a_1},{a_2}\}$, we have
\begin{alignat*}{3}
&\P(\ker \varphi_{x-{a}}) \cap S& &= \psi\left(\{a\}\times \P^1\right) \cup \psi\left(\P^1 \times \left\{\frac{h_0}{h_1}(a)^{1/q}\right\}\right).
\end{alignat*}
Since the polynomial $Q$ is irreducible, we have ${a_1} \neq {a_2}$. Furthermore, assuming that $Q$ is not a degenerate trap, we have $\frac{h_0}{h_1}({a_1}) \not\in k$, and thereby $\frac{h_0}{h_1}({a_1}) \neq \frac{h_0}{h_1}({a_2})$. Therefore $\P^1_Q \cap S$ is equal to
\begin{alignat*}{1}
\P(\ker \varphi_{x-{a_1}}) \cap \P(\ker \varphi_{x-{a_2}}) \cap S = \left\{\psi\left({a_1},\frac{h_0}{h_1}({a_2})^{1/q}\right),\psi\left({a_2},\frac{h_0}{h_1}({a_1})^{1/q}\right)\right\}.
\end{alignat*}
\end{proof}

In particular, when $Q$ is not a degenerate trap, $\P^1_Q$ is exactly the line passing through the two points $s_1 = \psi(a_1,b_1)$ and $s_2 = \psi(a_2,b_2)$. We get a $k$-isomorphism $\P^1 \rightarrow \P^1_Q : \alpha \mapsto s_1 - \alpha s_2$. For this reason the two points $s_1$ and $s_2$ play a central role in the rest of the analysis, and the following proposition shows that they behave nicely when $Q$ is a good polynomial.

\begin{prop}\label{prop:distinctroots}
Suppose $Q$ is a good polynomial. Then, $(\P^1_Q \cap S)(\algclosure) = \{s_1,s_2\}$, where $s_1 = (x-{a_1})(x-b_1)^q$, and $s_2 = (x-{a_2})(x-b_2)^q$, and the roots ${a_1}$, ${a_2}$, $b_1$ and $b_2$ are all distinct.
\end{prop}

\begin{proof}
From Lemma~\ref{lem:nondegenerate}, we can write $(\P^1_Q \cap S)(\algclosure) = \{s_1,s_2\}$ with $a_1 \neq a_2$ and $b_1 \neq b_2$.
If ${a_1} = b_2$ or ${a_2} = b_1$, then $Q$ divides $x^qh_1 - h_0$, a trap of level $0$.
Now, suppose ${a_1} = b_1$ (the case ${a_2} = b_2$ is similar). Since ${a_1}$ and ${a_2}$ are the two roots of~$Q$, and $Q$ divides $(x-{a_1})(h_0 - {a_1^q} h_1)$, then ${a_2}$ is a root of $h_0 - {a_1^q} h_1$. We get that $h_0({a_2}) = {a_1^q} h_1({a_2})$, so ${a_2}$ is a root of $h_1x^{q^{dm + 1}} - h_0$, a trap of level $dm$.
\end{proof}

\section{Irreducible covers of $\P^1_Q$} \label{sec:ramifications}
In this section we suppose that $Q$ is a good polynomial, and we consider the polynomials $s_1 = (x-{a_1})(x-b_1)^q$ and $s_2 = (x-{a_2})(x-b_2)^q$ as defined in Proposition~\ref{prop:distinctroots}, where ${a_1}$, ${a_2}$, $b_1$ and $b_2$ are all distinct. Consider the variety $\P^1_Q$ from~\eqref{eq:defnPQ}.\\

Recall that our goal is to prove that a significant proportion of the polynomials of $\P^1_Q(k)$ splits completely over $k$. As mentioned in Section~\ref{sec:overview}, our method consists in constructing a morphism $C \rightarrow \P^1_Q$ where $C$ is an absolutely irreducible curve defined over $k$, such that the image of any $k$-rational point of $C$ is a polynomial that splits completely over $k$. The absolute irreducibility is crucial as it implies that $C$ has a lot of $k$-rational points.
The idea is to consider the algebraic set
$$C = \{(u,r_1,r_2,r_3) \mid \text{the $r_i$'s are three distinct roots of $u$}\} \subset \P^1_Q \times \P^1 \times \P^1 \times \P^1,$$
and the canonical projection $C \rightarrow \P^1_Q$.

\begin{prop}\label{prop:Csplitcomp}
If $(u,r_1,r_2,r_3) \in C(k)$, then $u$ splits completely over $k$.
\end{prop}
\begin{proof}
Suppose that $(u,r_1,r_2,r_3)$ is a $k$-rational point of $C$. From Lemma~\ref{lem:pgl2orbit}, we get $u = g \star (x^q - x)$ where $g$ is the matrix $g \in \mathrm{PGL}_2(k)$ sending the three points $r_1,r_2$ and $r_3$ to $0$, $1$ and $\infty$.
In particular, the set of roots of $u$ is $g^{-1}(\P^1(\F_q))$ which are all in $\P^1(k)$.
\end{proof}
In the rest of this section, we prove that $C$ is absolutely irreducible (Proposition~\ref{prop:Cirred}).
The strategy is the following. Instead of considering directly $C$, which encodes three roots for each polynomial of $\P^1_Q$, we start with the variety
$$X = \{(u,r) \mid u(r) = 0\} \subset \P^1_Q \times \P^1,$$
which considers a single root for each polynomial. We can then ``add'' roots by considering fibre products.
Recall that given two covers $\nu : Z \rightarrow Y$ and $\mu: Z' \rightarrow Y$,
the geometric points of the fibre product $Z\times_YZ'$ are pairs $(z,z')$ such that $\nu(z) = \mu(z')$.
In particular, the fibre product over the projection $X \rightarrow \P^1_Q$~is
\begin{alignat*}{1}
X\times_{\P^1_Q} X &= \{((u_1,r_1),(u_2,r_2) ) \mid u_1(r_1) = 0, u_2(r_2) = 0, u_1 = u_2\}\\
&\cong \{(u,r_1,r_2) \mid u(r_1) = 0, u(r_2) = 0\}.
\end{alignat*}
This product $X\times_{\P^1_Q} X$ contains a trivial component, the diagonal, corresponding to triples $(u,r,r)$. The rest is referred to as the non-trivial part, and we prove that it is an absolutely irreducible curve (Corollary~\ref{lem:Y2Irred}).
Iterating this construction, the fibre product $(X\times_{\P^1_Q}X)\times_{X}(X\times_{\P^1_Q}X)$ (over the projection $X\times_{\P^1_Q}X \rightarrow X$ to the first component) encodes quadruples $(u,r_1,r_2,r_3)$. Therefore the curve $C$ naturally embeds into the non-trivial part of this product.
We prove that this non-trivial part is itself an absolutely irreducible curve (Lemma~\ref{lem:x3irred}).\\

Instead of the projection $X \rightarrow \P^1_Q$, we work with an isomorphic cover $\theta$.
It is easy to see that the canonical projection $X \rightarrow \P^1$ is an isomorphism, with inverse $r \mapsto (s_2(r)s_1 - s_1(r)s_2,r)$.
Through the isomorphisms ${X \cong \P^1}$ and ${\P^1_Q \cong \P^1}$, this projection is isomorphic to the cover $\theta$ in the following commutative diagram (where, again, the morphisms are written affinely for convenience):
\begin{equation*}
\xymatrixrowsep{0.3pc}
\xymatrix{
&(u,r)\ar@{|->}[r] & u &\\
(u,r)\ar@{|->}[ddd]&X \ar[r]\ar[ddd]^{\wr} & \P^1_Q\ar[ddd]^{\wr} & s_1 - \alpha s_2\ar@{|->}[ddd]\\
&&  &\\ 
&&  &\\ 
r&\P^1 \ar[r]^{\theta} & \P^1 & \alpha\\
&r\ar@{|->}[r] & s_1(r)/s_2 (r). &
}
\end{equation*}
For convenience, consider $\theta$ as a cover $X_1 \rightarrow X_0$ where $X_0 = X_1 = \P^1$.
As a first step, we study the induced fibre product ${X_1\times_{X_0}X_1}$.
It contains the diagonal $\Delta_1$, isomorphic to $X_1$.
We wish to show that $Y_2 = X_1\times_{X_0}X_1 \setminus \Delta_1$ is absolutely irreducible.
The second step consists in showing that $X_2 \times_{X_1} X_2 \setminus \Delta_2$ is also absolutely irreducible, where $X_2$ is a desingularisation of $Y_2$ and $\Delta_2$ is the diagonal.
The following lemma provides a general method used in both steps.

\begin{lem}\label{lem:FiberProductIrred}
Let $Y$ and $Z$ be two absolutely irreducible, smooth, complete curves over $k$, and consider a cover $\eta : Z \rightarrow Y$. If there exists a point $a \in Z$ such that $\eta$ is not ramified at $a$ and $\#(\eta^{-1}(\eta(a))) = 2$, then $Z\times_YZ \setminus \Delta$ is absolutely irreducible, where $\Delta$ is the diagonal component.
\end{lem}

\begin{proof}
By contradiction, suppose that $Z\times_YZ \setminus \Delta$ is not absolutely irreducible, and can be decomposed as two components ${A \cup B}$. 
Let $\mathrm{pr} : Z\times_YZ \rightarrow Z$ be the projection on the first factor.
Since $Z\times_YZ$ is complete, both $A$ and $B$ are complete,
so we have $\mathrm{pr}(A) = \mathrm{pr}(B) = \mathrm{pr}(\Delta) = Z$. Observe that $\mathrm{pr}^{-1}(a)$ consists of $\#(\eta^{-1}(\eta(a))) = 2$ points, so one of them must belong to two of the components $A$, $B$ and $\Delta$. That point must therefore be singular in $Z\times_YZ$, contradicting the fact that $\eta$ is not ramified at $a$ (recall that a point $(z_1,z_2) \in Z\times_YZ$ is singular if and only if $\eta$ is ramified at both $z_1$ and $z_2$).
\end{proof}

\begin{cor}\label{lem:Y2Irred}
The curve $Y_2 = X_1\times_{X_0}X_1 \setminus \Delta_1$ is absolutely irreducible.
\end{cor}

\begin{proof}
First observe that $\theta$ is ramified only at $b_1$ and $b_2$ (as can be verified from the explicit formula $\theta(r) = s_1(r)/s_2 (r)$). In particular, it is not ramified at $a_1$. Since $\#(\theta^{-1}(\theta(a_1))) = \#\{a_1,b_1\} = 2$, we apply Lemma~\ref{lem:FiberProductIrred}.
\end{proof}

\begin{lem}\label{lem:desingbij}
The desingularisation morphism $\nu : X_2 \rightarrow Y_2$ is a bijection between the geometric points.
\end{lem}

\begin{proof}
It is sufficient to prove that for any singular point $P$ on $Y_2$, and $\varphi: \tilde Y_2 \rightarrow Y_2$ the blowing-up at $P$, the preimage $\varphi^{-1}(P)$ consists of a single smooth point.
Up to a linear transformation of $X_1 = \P^1$, we can assume that $s_1$ and $s_2$ are of the form $s_1(x) = (x-1)x^q$ and $s_2(x) = x-a$, for some $a \neq 0,1$. 
The intersection $A$ of the curve $Y_2$ with the affine patch $\mathbf A^2 \subset \P^1 \times \P^1$ is then defined by the polynomial 
$$f(x,y) = \frac{s_1(x)s_2(y) - s_1(y)s_2(x)}{x-y} = \frac{x^q(x-1)(y-a) - y^q(y-1)(x-a)}{x-y}.$$
It remains to blow up $A$ at the singularity $(0,0)$ (which corresponds to $(b_1,b_1)$ through the linear transformation), and check the required properties. This is easily done following~\cite[Ex. 4.9.1]{hartshorne}, and we include details for the benefit of the reader. Let $\psi: Z \rightarrow \mathbf A^2$ be the blowing-up of $\mathbf A^2$ at $(0,0)$. The inverse image of $A$ in $Z$ is defined in $\mathbf A^2 \times \P^1$ by the equations $f(x,y) = 0$ and $ty = xu$ (where $t$ and $u$ parameterize the factor $\P^1$). It consists of two irreducible components: the blowing-up $\tilde A$ of $A$ at $(0,0)$ and the exceptional curve $\psi^{-1}(0,0)$. Suppose $t \neq 0$, so we can set $t = 1$ and use $u$ as an affine parameter (since $f$ is symmetric, the case $u \neq 0$ is similar). We have the affine equations $f(x,y) =0$ and $y = xu$, and substituting we get $f(x,xu) = 0$, which factors as
$$f(x,xu) = x^{q-1} \frac{(x-1)(xu-a) - u^q(xu-1)(x-a)}{1-u}.$$
The blowing-up $\tilde A$ is defined on $t = 1$ by the equations $g(x,u) = f(x,xu)/x^{q-1} = 0$ and $y = xu$. It meets the exceptional line only at the point $u = 1$, which is non-singular.
\end{proof}

The projection $X_1\times_{X_0}X_1 \rightarrow X_1$ on the first component induces another cover ${\theta_2 : X_2 \rightarrow X_1}$, through which we build the fibre product $X_2\times_{X_1}X_2$. As above, it contains a diagonal component $\Delta_2$ isomorphic to $X_2$.

\begin{lem}\label{lem:x3irred}
The curve $Y_3 = X_2\times_{X_1}X_2 \setminus \Delta_2$ is absolutely irreducible.
\end{lem}

\begin{proof}
Let $\nu : X_2 \rightarrow Y_2$ be the bijective morphism from Lemma~\ref{lem:desingbij}. 
Since $\theta_1$ is only ramified at $b_1$ and $b_2$,
the cover $\theta_2$ is ramified at most at the points $\nu^{-1}(b_i,b_i)$ and $\nu^{-1}(a_i,b_i)$ (for $i \in \{1,2\}$).
In particular, it is not ramified at $\nu^{-1}(b_1,a_1)$. Since $\#(\theta_2^{-1}(\theta_2(\nu^{-1}(b_1,a_1)))) = \#\{\nu^{-1}(b_1,a_1),\nu^{-1}(b_1,b_1)\} = 2$, we apply Lemma~\ref{lem:FiberProductIrred}.
\end{proof}

\begin{prop}\label{prop:Cirred}
The curve $C$ is absolutely irreducible.
\end{prop}
\begin{proof}Let $\nu : X_2 \rightarrow Y_2$ be the morphism from Lemma~\ref{lem:desingbij}.
It is an isomorphism away from the singularities of $Y_2$, so
$$C \longrightarrow Y_3 : (u,r_1,r_2,r_3) \longmapsto (\nu^{-1}(r_1,r_2), \nu^{-1}(r_1,r_3))$$
is a morphism. It is an embedding, and the result follows from Lemma~\ref{lem:x3irred}.
\end{proof}

\section{Counting split polynomials in $\P^1_Q$}\label{sec:proofmainprop}

Recall that we wish to prove Proposition~\ref{prop:onestepdescent} by showing that $\P^1_Q(k)$ contains a lot of polynomials that split into good polynomials over $k$.
The results of Section~\ref{sec:ramifications} allow us to prove in Theorem~\ref{thm:splitpolyonp} that a lot of polynomials in~$\P^1_Q(k)$ do split. We then show in Proposition~\ref{prop:coprimality} that all these polynomials are coprime, which implies that bad polynomials cannot appear too often.

\begin{thm}\label{thm:splitpolyonp}
Let $k/\F_{q^d}$ be a field extension of degree $m$, and $Q$ be a good irreducible quadratic polynomial in $k[x]$ coprime to $h_1$.
If $dm \geq 23$, there are at least $\#k/2q^3$ polynomials in~$\P^1_Q$ that split completely over the field $k$.
\end{thm}

\begin{proof}
Let ${\Theta : Y_3 \rightarrow \P^1_Q}$ be the cover resulting from the composition of the successive covers of Section~\ref{sec:ramifications}.
Let $S_3 = \Theta^{-1}(\P^1_Q \cap S)$. 
The embedding $C\rightarrow Y_3$ from Proposition~\ref{prop:Cirred} has image $Y_3 \setminus S_3$.
The morphism $$\mu : Y_3 \rightarrow \P^1 \times \P^1 \times \P^1 : (\nu^{-1}(r_1,r_2), \nu^{-1}(r_1,r_3)) \mapsto (r_1,r_2,r_3)$$
restricts to an embedding of $Y_3\setminus S_3$.
Let $A$ be the intersection of $\mu(Y_3)$ with the affine patch $\mathbf A^3$. The curve $A$ is a component of the (reducible) curve defined by the equations $\theta(r_1) = \theta(r_2)$ and $\theta(r_1) = \theta(r_3)$. Therefore $A$ is of degree at most $4(q+1)^2$. If $B$ is the closure of $A$ in $\P^3$, then~\cite[Th. 3.1]{Ba96} shows that
$$|\# B(k) - \#k - 1| \leq 16(q+1)^4\sqrt{\#k}.$$
Since $Y_3$ is complete, $\mu(Y_3)$ is closed, so all the points of $B \setminus A$ are at infinity, and there are at most $\deg(B) \leq 4(q+1)^2$ of them. Also, at most $2(q^3-q)$ points of $B$ are in $\mu(S_3)$ (because $\#S = 2$ and $\Theta$ is of degree $q^3-q$). Therefore
$$\#C(k) = \#(Y_3\setminus S_3)(k) \geq \#k + 1 - 16(q+1)^4\sqrt{\#k} - 4(q+1)^2 - 2(q^3-q).$$
Since $q \geq 2$ and $dm \geq 23$, we get $\#C(k) \geq \#k /2$. From Proposition~\ref{prop:Csplitcomp}, and the fact that the map $\Theta$ is $q^3-q$ to one, we get that at least $\#k/2q^3$ polynomials in $\P^1_Q$ split completely over $k$.
\end{proof}

Let $\varphi$ be the morphism defined in~\eqref{eq:morphxqtoh}.

\begin{prop}\label{prop:coprimality}
Suppose $Q$ is a good polynomial. For any two distinct polynomials $f$ and $g$ in $\P^1_Q(\algclosure)$, we have $\gcd(f,g) = 1$ and $\gcd(h_1\varphi(f),h_1\varphi(g)) = Q$.
\end{prop}
\begin{proof}
Let $s_1$ and $s_2$ be as in Proposition~\ref{prop:distinctroots}. They have no common root.
Since $f$ and $g$ are distinct, all the polynomials of $\P^1_Q$ are of the form $\alpha f + \beta g$ for $(\alpha:\beta) \in \P^1$. Then, if $r$ is a root of $f$ and $g$, $r$ is a root of all the polynomials of $\P^1_Q$. In particular, it is a root of both $s_1$ and $s_2$, a contradiction. This shows that $\gcd(f,g) = 1$.

Similarly, if a polynomial $h$ divides $h_1\varphi(f)$ and $h_1\varphi(g)$, it must also divide both
$h_1\varphi(s_1) = (x-a_1)(h_0 - b_1^qh_1) \text{, and }
h_1\varphi(s_2) = (x-a_2)(h_0 - b_2^qh_1).$
Since $h_0 - b_1^qh_1$ and $h_0 - b_2^qh_1$ are coprime, $h$ must divide $Q$.
\end{proof}

\subsection*{Proof of Proposition~\ref{prop:onestepdescent}} \label{sec:proofprop}
As discussed in Section~\ref{sec:overview}, it is sufficient to prove that a uniformly random element of $\P^1_Q(k)$ has a good probability to lead to an elimination into good polynomials.
A polynomial $f \in \P^1_Q(k)$ leads to an elimination into good polynomials if $f$ splits completely over $k$ into good linear polynomials, and $\varphi(f)$ is itself a good polynomial.

Let $A$ be the set of polynomials of $\P^1_Q(k)$ that split completely over $k$.
From Theorem~\ref{thm:splitpolyonp}, $A$ contains at least $q^{dm-3}/2$ elements.
Trap roots $\tau$ occurring in $A$ or $\varphi(A)$ must be roots of $h_1x^q - h_0$, or of $h_1x^{q^{dn+1}}-h_0$ for $n \mid m/2$, or satisfy $\frac{h_0}{h_1}(\tau) \in \F_{q^{dm/2}}$. There are at most $q^{\frac{dm}{2} + 3}$ such trap roots. From Proposition~\ref{prop:coprimality}, any trap root can only occur once in $A$ and in $\varphi(A)$.
So there are at most $2q^{\frac{dm}{2} + 3}$ polynomials in $A$ for which trap roots appear.
Therefore the number of elements in $A$ leading to a good reduction is at least $$\frac{1}{2}q^{dm-3} - 2q^{\frac{dm}{2} + 3} \geq \frac{1}{2}\left(q^{dm-3} - 4q^{dm - 8}\right) \geq \frac 1 4 q^{dm-3},$$
using $dm \geq 23$. Since $\P^1_Q(k)$ contains $q^{dm}+1$ elements, the probability of a random element to lead to a good elimination is $1/O(q^3)$.\qed

\section*{Acknowledgements}
Part of this work was supported by the Swiss National Science Foundation
under grant number 200021-156420.

\bibliographystyle{amsalpha}
\bibliography{biblio}

\end{document}